\documentclass{article}
\usepackage{a4wide}
\usepackage{amsthm,amsfonts,amsmath}
\usepackage{microtype}

\title{Asymptotic upper bounds on progression-free sets in $\Z_p^n$}
\author{Dion Gijswijt}

\newcommand*{\Z}{\mathbb{Z}}
\newcommand*{\F}{\mathbb{F}}
\newcommand*{\N}{\mathbb{N}}

\newcommand*{\transp}{\mathsf{T}}
\newcommand*{\nom}[3]{\left(\begin{smallmatrix}#1\\#2\end{smallmatrix}\right)_{#3}}

\DeclareMathOperator{\rank}{rank}
\DeclareMathOperator{\linspan}{span}

\newtheorem{theorem}{Theorem}
\newtheorem{lemma}{Lemma}
\newtheorem{proposition}{Proposition}
\newtheorem{corollary}{Corollary}
\newtheorem{remark}{Remark}

\begin{document}
\maketitle
\begin{abstract}We show that any subset of $\Z_p^n$ ($p$ an odd prime) without $3$-term arithmetic progression has size $O(p^{cn})$, where $c:=1-\frac{1}{18\log p}<1$. In particular, we find an upper bound of $O(2.84^n)$ on the maximum size of an affine cap in $GF(3)^n$.
\end{abstract}
\section*{Introduction}
Given an abelian group $G$, a subset $A\subseteq G$ is \emph{progression-free} if there are no disctinct $a,b,c\in A$ for which $a+b=2c$.  In their recent paper \cite{CLP}, Croot, Lev and Pach used the polynomial method to show an upper bound of $O(4^{0.926\cdot n})$ on the size of progression-free sets in $\Z_4^n$. In this paper, we extend their method to progression-free sets in $\Z_p^n$, where $p$ is an odd prime. This improves the bound $O(\frac{p^n}{n})$ of Meshulam \cite{Meshulam} and the bound $O(\frac{3^n}{n^{1+\epsilon}})$ (where $\epsilon >0$ is a constant) in the case $p=3$ due to Bateman and Katz \cite{Bateman}.

\begin{remark}
While submitting the paper, the author was informed that Jordan S. Ellenberg proved a similar result \cite{Ellenberg} three days ago. In their paper an upper bound of $O(2.756^n)$ for progression-free subsets of $\Z_3^n$ (and hence affine caps in $\F_3^n$ is proved. The paper also claims that their method gives an upper bound of $O(p^{cn})$ for some $c=c(p)<1$ in the case of $\Z_p^n$.

\noindent \textbf{Update v2:} Since the arguments of our two papers were essentially identical, we decided to publish our solutions as a joint paper \cite{capsets}.
\end{remark}

\section*{Main theorem}
Throughout, $\F=GF(p)$ will be a finite field, where $p$ is an odd prime. We denote by $L_n:=\linspan\{x^\alpha: \alpha\in\{0,1\ldots, p-1\}^n\}$ the linear space of polynomials over $\F$ in $n$ variables in which no variable occurs with exponent more than $p-1$. Here we use the notation $x^\alpha:=x_1^{\alpha_1}x_2^{\alpha_2}\cdots x_n^{\alpha_n}$. Also, we denote $|\alpha|:=\alpha_1+\cdots+\alpha_n$. For $f\in L_n$, we denote by $Z(f):=\{a\in \F^n\mid f(a)=0\}$ the zero set of $f$. For any integer $d\in \{0,\ldots, (p-1)n\}$ we denote by $L_{n,d}$ the subspace of $L_n$ consisting of polynomials of degree at most $d$. Observe that $\dim L_{n,d}+ \dim L_{n,(p-1)n-d-1}=p^n$ since the map $(\alpha_1,\ldots, \alpha_n)\mapsto (p-1-\alpha_1,\ldots, p-1-\alpha_n)$ induces a bijection from the set of monomials in $L_n$ to itself. We will use the following estimate on the dimension of $L_{n,(p-1)n/3}$.

In order to bound the dimension of the subspaces $L_{n,d}$, we use the following inequality.
\begin{theorem}[Hoeffding inequality \cite{Hoeffding}]
Let $X_1,\ldots, X_n$ be independent random variables on $[a_i, b_i]$ and let $S=X_1+\cdots+X_n$. Then 
$$
\Pr(E[S]-S\geq t) \leq \exp\left(-\frac{2t^2}{\sum_{i=1}^n(b_i-a_i)^2}\right).
$$
\end{theorem}

\begin{lemma}\label{entropy}
Let $c:=1-\frac{1}{18\log p}<1$. For $n$ a positive multiple of $3$, we have $\dim L_{n,(p-1)n/3}\leq p^{cn}$.
\end{lemma}
\begin{proof}
Denote by $\nom{n}{k}{p-1}:=|\{a\in \{0,1,\ldots, p-1\}^n: |a|=k\}|$ the extended binomial coefficients (see e.g. \cite{Fahssi}). So we have $\dim L_{n,d}=\sum_{k=0}^d \nom{n}{k}{p-1}$. Let $X_1,\ldots, X_n$ be i.i.d. random variables with $\Pr[X_i=t]=\tfrac{1}{p}$ for $t=0,\ldots, p-1$. Let $S:=X_1+\cdots+X_n$. It is easy to see that $\nom{n}{k}{p-1}=p^n\Pr[X=k]$. The expected value of $S$ equals $\tfrac{1}{2}(p-1)n$.

By Hoeffding's inequality, we have 
$$
\Pr[S\leq\tfrac{1}{3}(p-1)n]=\Pr[S\leq\tfrac{1}{2}(p-1)n-\tfrac{1}{6}(p-1)n]\leq e^{-\frac{1}{18}n}.
$$ 
It follows that 
$$
\dim L_{n,(p-1)n/3}\leq p^n\cdot e^{-\frac{1}{18}n}=p^{n\cdot (1-\frac{1}{18\log p})}=p^{cn}.
$$
\end{proof}

\begin{proposition}\label{bijection}
The evaluation map $\phi:L_n\to \F^{\F^n}$ given by $\phi(f)=(f(a))_{a\in \F^n}$ is a linear bijection. 
\end{proposition} 
\begin{proof}
The fact that $\phi$ is linear is clear. Since $\dim L_n=\F^n=\dim \F^{\F^n}$, it suffices to show that $\phi$ is injective. We will show this by induction on $n$. If $n=1$, this follows since a nonzero polynomial $f=c_0+c_1x_1+\cdots+c_{p-1}x_1^{p-1}$ has at most $p-1<p$ roots in $\F$. 

Now let $n\geq 2$ and let $f\in L_n$ be such that $Z(f)=\F^n$. We need to show that $f=0$. Write $f=f_0+x_nf_1+x_n^2f_2+\cdots+x_n^{p-1}f_{p-1}$, where $f_0,\ldots, f_{q-1}\in L_{n-1}$. Observe that for any $a_1,\ldots, a_{n-1}\in \F$ the univariate polynomial $g(x_n):=\sum_{i=0}^{p-1} x_n^i\cdot f_i(a_1,\ldots, a_{n-1})$ evaluates to zero on the whole of $\F$ and therefore is the zero polynomial. That is, for all $a_1,\ldots, a_{n-1}$ and all $i=0,\ldots, p-1$ we have $f_i(a_1,\ldots, a_{n-1})=0$. By induction it follows that $f_i=0$ for $i=0,\ldots, p-1$ and hence that $f=0$.   
\end{proof}

\begin{lemma}\label{rank}
Let $g=\sum_{\alpha,\beta}C_{\alpha,\beta}\, x^\alpha y^\beta\in \F[x_1,\ldots,x_n,y_1,\ldots, y_m]$, where $C\in \F^{\N^n\times\N^m}$. Let $A\subseteq \F^n$ and $B\subseteq \F^m$. Define the matrix $M\in \F^{A\times B}$ by $M_{ab}:=g(a,b)$. Then $\rank M\leq \rank C$.
\end{lemma}
\begin{proof}
Let $M_A\in \F^{\N^n\times A}, M_B\in \F^{\N^m\times B}$ be defined by $(M_A)_{\alpha,a}:=a^\alpha$ and $(M_B)_{\beta, b}:=b^\beta$. It is easy to check that $M:=M_A^\transp C M_B$. Hence, $\rank M\leq \rank C$.
\end{proof}

\begin{proposition}\label{diagonal}
Let $f\in L_{n,2d}$ and let $A\subseteq\F^n$. Suppose that for all $a,b\in A$ we have: $f(a+b)=0$ if and only if $a\neq b$. Then $|A|\leq 2\dim L_{n,d}$. 
\end{proposition}
\begin{proof}
Let $g\in \F[x_1,\ldots,x_n,y_1,\ldots, y_n]$ be defined by $g(x,y):=f(x+y)$. So $g$ has degree at most $2d$. Write $g=\sum_{\alpha,\beta}C_{\alpha,\beta}x^\alpha y^\beta$. Note that $C_{\alpha,\beta}$ is nonzero only if $|\alpha|\leq d$ or $|\beta|\leq d$. It follows that the support of $C$ is contained in the union of the rows indexed by monomials of degree at most $d$ and the columns indexed by monomials of degree at most $d$. Hence, $\rank C\leq 2\dim L_{n,d}$.

On the other hand, the $A\times A$ matrix $M$ defined by $M_{a,b}:=g(a,b)$ is a diagonal matrix with nonzero diagonal elements and therefore has rank $|A|$. By Lemma \ref{rank}, it follows that $|A|=\rank M\leq \rank C\leq 2\dim L_{n,d}$.
\end{proof}

\begin{theorem}[Main theorem]\label{main} Let $c:=1-\frac{1}{18\log p}<1$. For $A\subseteq \F^n$ progression free, we have $|A|=O(p^{cn})$.
\end{theorem}
\begin{proof}
Let $n$ be a multiple of $3$ and let $A\subseteq \F^n$ be progression free. It suffices to show that $|A|\leq 3p^{cn}$.

Define $B:=\{a+b\mid a,b\in A\text{ with $a\neq b$}\}$ and $C:=\{a+a\mid a\in A\}$. Since $A$ is progression-free we have $B\cap C=\emptyset$. Let 
\begin{eqnarray*}
K&:=&\{f\in L_n\mid (\F^n\setminus C)\subseteq Z(f)\},\\
L&:=&L_{n, \tfrac{2}{3}(p-1)n}.
\end{eqnarray*}
Note that $K$ is a linear space of dimension $|C|$ by Proposition \ref{bijection}. By Lemma \ref{entropy},  $L$ is a linear space of dimension 
$$
\dim L\geq p^n-\dim L_{n,\tfrac{1}{3}(p-1)n-1}\geq p^n-p^{cn}.
$$ 
Denote $V:=K\cap L$. We have 
\begin{equation}\label{dimV}
\dim V\geq \dim L+\dim K-p^n\geq |C|-p^{cn}.
\end{equation}
In particular, we may assume that $V$ has positive dimension, for otherwise $|A|=|C|\leq p^{cn}$, and we would be done.

By Proposition \ref{bijection}, we can view $V$ as a linear subspace of $\F^C$. Hence, there is a subset $C'\subseteq C$ of size $\dim V$ such that the evaluation map $\phi: V\to \F^{C'}$ given by $\phi(f):=(f(c))_{c\in C'}$ is surjective. Hence, we can choose $f\in V$ such that $f(c)=1$ for all $c\in C'$.

Let $A':=\{a\in A\mid a+a\in C'\}$. Since $p$ is odd, we have $|A'|=|C'|$. Since $f\in K$, we have $B\subseteq (\F^n\setminus C)\subseteq Z(f)$. This implies that $f(a+b)=0$ for all $a,b\in A'$ distinct. By our choice of $f$ we also have $f(a+a)=1$ for all $a\in A'$. Since $f$ has degree at most $\tfrac{2}{3}(p-1)n$, Proposition \ref{diagonal} implies that $|A'|\leq 2\dim L_{n,\tfrac{1}{3}(p-1)n}=2\dim L$. Hence, $|C'|=|A'|\leq 2p^{cn}$. 
By (\ref{dimV}), we obtain $$|A|=|C|\leq p^{cn}+\dim V=p^{cn}+|C'|\leq 3p^{cn}.$$ 
\end{proof}

In the special case $p=3$, progression-free sets correspond exactly to affine caps. The best known \emph{lower} bound for affine caps in $\F_3^n$ is $\Omega(2.2174^n)$ due to Edel \cite{Edel}. Since $3^{1-\tfrac{1}{18\log 3}}=2.84\cdot$, Theorem \ref{main} implies an \emph{upper} bound of $O(2.84^n)$, improving the previous best upper bound of $O(\frac{3^n}{n^{1+\epsilon}})$ due to Bateman and Katz \cite{Bateman}.
\begin{corollary} 
The maximum size of an affine cap in $\F_3^n$ is $O(2.84^n)$.
\end{corollary}

\section*{Acknowledgements}
The author would like to thank Jop Bri\"et, Viresh Patel, Guus Regts and Jeroen Zuiddam for very useful discussions on the polynomial method, which ultimately led to the current paper.

\end{document}